\documentclass[11pt]{amsart}
\newtheorem*{thmA}{Theorem A}\newtheorem*{thmB}{Theorem B}\newtheorem*{thmC}{Theorem C}
\newtheorem{thm}{Theorem}[section]\newtheorem{lem}[thm]{Lemma}\newtheorem{prop}[thm]{Proposition}

\begin{document}
\title{Supports of irreducible characters of $p$-groups}
\author{Tom Wilde}

\begin{abstract}
If $\chi$ is an irreducible character of a finite group $G$ then the support of $\chi,$ denoted $\mathrm{Supp}(\chi),$ is the set of $g\in G$ with $\chi(g)\neq 0.$ In this note, we study the supports of characters of certain types of $p$-groups. We show that if $\chi\in\mathrm{Irr}(P)$ where $P$ is a metabelian $p$-group, then $\mathrm{Supp}(\chi)$ contains at most $|P|/\chi(1)^2$ conjugacy classes, and we give a bound on the orders of elements in $\mathrm{Supp}(\chi).$ We give more precise results for $p$-groups of nilpotence class at most $3$ and groups of order dividing $p^9.$
\end{abstract}

\maketitle

\section{Introduction}
Let $\mathrm{Irr}(G)$ denote the set of complex irreducible characters of a finite group $G.$ If $\chi\in\mathrm{Irr}(G)$ then the support of $\chi$ is the subset of $G$ on which $\chi$ does not vanish. We write $\mathrm{Supp}(\chi)$ to denote the support of $\chi.$ In this note, we present some results about the supports of irreducible characters of $p$-groups. (A $p$-group is a finite group whose order is a power of the prime $p.$) 

If $\chi$ is an irreducible character of a $p$-group $P,$ then $\chi$ is induced from a linear character of a subgroup $H\subseteq P,$ where $|H|=|P|/\chi(1).$ (See Corollary $6.14$ in \cite{Isaacs1976}.) Then every element of $\mathrm{Supp}(\chi)$ has a conjugate that is contained in $H,$ and it follows that $\mathrm{Supp}(\chi)$ contains at most $|P|/\chi(1)$ conjugacy classes, and also that every element of $\mathrm{Supp}(\chi)$ has order dividing $|P|/\chi(1).$ We show that these bounds can be strengthened for some interesting classes of $p$-groups. Our first result concerns metabelian $p$-groups.

\begin{thmA}
Let $P$ be a metabelian $p$-group and let $\chi\in\mathrm{Irr}(P).$ Then 
\begin{enumerate}
\item $\mathrm{Supp}(\chi)$ contains at most $|P|/\chi(1)^2$ conjugacy classes.
\item If $p$ is odd, each element of $\mathrm{Supp}(\chi)$ has order at most $|P|/\chi(1)^{3/2}.$
\end{enumerate}
\end{thmA}

Theorem A$(1)$ is false for $p$-groups in general. Using GAP (\cite{GAP}), we found that the group $P$ with GAP identifier $\text{SmallGroup}(512,2015)$ has two characters of degree $8$ that are non-vanishing on $10$ conjugacy classes of $P.$ Since $10\times 64>512,$ this violates the bound of Theorem A$(1).$ The group $P$ has derived length $3.$ 

The exponent of $2$ in Theorem A$(1)$ cannot be improved, but in the case of Theorem A$(2),$ we do not know the best value of the exponent. Theorems B$(4)$ and C below show that its correct value is $2$ in some cases. Since groups of nilpotence class at most $3$ are metabelian, Theorem B can be seen as a strong version of Theorem A for such groups.

\begin{thmB}
Let $P$ be a $p$-group of nilpotence class at most $3.$ Let $\chi\in\mathrm{Irr}(P)$ and $x\in\mathrm{Supp}(\chi).$ Let $h(x)$ be the size of the conjugacy class of $x.$ Then
\begin{enumerate}
\item If $\chi$ is faithful, then $\frac{\chi(x)\sqrt{h(x)}}{\chi(1)}$ is a root of unity.
\item If $\chi$ is faithful, $\mathrm{Supp}(\chi)$ contains exactly $|P|/\chi(1)^2$ conjugacy classes.
\item If $p$ is odd, then $x^m\in\mathrm{Supp}(\chi)$ for all $m\in\mathbb Z.$
\item If $p$ is odd, then the order of $x$ divides $|P|/\chi(1)^2.$
\end{enumerate}
\end{thmB}

We remark that in the situation of Theorem B, often $\mathrm{Supp}(\chi)=\mathbf Z(\chi),$ where $\mathbf Z(\chi)/\ker(\chi)$ is the centre of $P/\ker(\chi)$. In other words, $\chi$ is often of central type when regarded as a character of $G/\ker(\chi).$ This always occurs when $P$ has nilpotence class $2,$ by for example Lemma 2.27(f) and Theorem 2.31 in \cite{Isaacs1976}. (In general, groups whose irreducible characters $\chi$ all satisfy $\mathrm{Supp}(\chi)=\mathbf Z(\chi)$ are called GVZ groups, and are studied in \cite{Nenciu2012}.) Theorem B is trivial for such characters, but of course, not all characters of groups of nilpotence class $3$ have this property. For example, if $P$ is the dihedral group of order $16,$ $x\in P$ is an element of order $8$ and $\chi\in\mathrm{Irr}(P)$ is faithful then $x$ is not central but $\chi(x)=\pm\sqrt 2.$ Since $\chi(1)=2$ and $\chi(x^2)=0,$ this example also shows that $p$ does need to be odd for parts $(3)$ and $(4)$ of Theorem B.

We mentioned that we do not know what value is the best possible for the exponent in Theorem A$(2).$ Theorem B shows that the correct exponent is $2$ for $p$-groups of nilpotence class at most $3,$ when $p$ is odd. Our last result shows that the same holds for all $p$-groups of order at most $p^9,$ where $p$ is odd (we do not know what happens for $p$-groups of larger order).

\begin{thmC}
Let $P$ be a $p$-group, where $p$ is odd, and let $\chi\in\mathrm{Irr}(P).$ Suppose there exists $x\in\mathrm{Supp}(\chi)$ such that the order of $x$ does not divide $|P|/\chi(1)^2.$ Then $\chi(1)$ is divisible by $p^4$ and $|P|$ is divisible by $p^{10}.$
\end{thmC}

The proofs of Theorems A, B and C are in Sections $2,3$ and $4$ respectively. As far as we know, these results have not been reported before, but similar questions have been considered in \cite{Chillag1999} and \cite{Moreto2004}, for example. In particular, we draw attention to Theorem C in \cite{Moreto2004}, which looks at the situation from the opposite point of view, giving a lower bound for the number of conjugacy classes on which an irreducible character of an arbitrary $p$-group vanishes.

\section{Metabelian $p$-groups}
In this section, we prove Theorem A. We need the following well-known fact about metabelian groups.

\begin{lem}\label{nM}
Let $G$ be a finite metabelian group and let $\chi\in\mathrm{Irr}(G).$ Then $G$ has a normal subgroup $K$ with $G/K$ abelian, such that $\chi=\lambda^G$ for some linear character $\lambda\in\mathrm{Irr}(K).$
\end{lem}

\begin{proof}
This follows from Corollary 2.6 in \cite{Isaacs1964}, taking the subgroup $N$ in that corollary to be the derived subgroup of $G.$
\end{proof}

We begin with the proof of Theorem A$(1).$ If $A$ and $H$ are groups with $H$ acting on $A$ (we always mean an action via automorphisms of $A$), then for $\lambda\in\mathrm{Irr}(A),$ we define the trace of $\lambda$ with respect to $H,$ denoted $\mathrm{tr}_H(\lambda),$ by $$\mathrm{tr}_H(\lambda)=\sum_{h\in H}\lambda^h.$$ Equivalently, $\mathrm{tr}_H(\lambda)=(\lambda^{AH})_A,$ where $AH$ is the semidirect product of $A$ and $H$ corresponding to the given action and $\lambda^{AH}$ is the induced (not necessarily irreducible) character. Also define $$\sigma(A,\lambda,H)=\mathop{\sum_{x\in A\text{ with }}}_{\mathrm{tr}_H(\lambda)(x)\neq 0}|{C_H(x)}|.$$ We prove Theorem A$(1)$ via the following result:

\begin{prop}\label{sigmaineq}
Suppose $A$ and $H$ are abelian $p$-groups for some prime $p,$ with $H$ acting on $A.$ Let $\lambda\in\mathrm{Irr}(A).$ Then $\sigma(A,\lambda,H)\leq | A||C_H(\lambda)|.$
\end{prop}

Our proof of Proposition \ref{sigmaineq} requires two lemmas.

\begin{lem}\label{sigmalem1}
Suppose $A$ and $H$ are $p$-groups with $A$ abelian and $H$ acting on $A.$ Let $\lambda\in\mathrm{Irr}(A)$ be a character of $A.$ Let $B$ be a $H$-invariant subgroup of $A$ with $|A:B|=p.$ Then one of the following cases holds:
\begin{enumerate}
\item $C_H(\lambda_B)=C_H(\lambda).$
\item $|C_H(\lambda_B)|=p|C_H(\lambda)|$ and $\mathrm{tr}_H(\lambda)(x)=0$ for all $x\in A\backslash B.$
\end{enumerate}
\end{lem}

\begin{proof}
Let $\mathrm{Irr}(A\vert\lambda_B)$ denote the set of irreducible characters of $A$ that lie over $\lambda_B\in\mathrm{Irr}(B).$ Then $C_H(\lambda_B)$ acts on $\mathrm{Irr}(A\vert\lambda_B).$ If the action fixes $\lambda$ then case $(1)$ holds. If not, then since $|C_H(\lambda_B):C_H(\lambda)|\leq|\mathrm{Irr}(A\vert\lambda_B)|=p$ and $H$ is a $p$-group, $|C_H(\lambda_B):C_H(\lambda)|=p.$ This is case $(2),$ but we must also show that $\mathrm{tr}_H(\lambda)$ vanishes outside of $B$ in this case. Fix $g\in C_H(\lambda_B)$ with $g\notin C_H(\lambda),$ so $\lambda^g=\lambda\mu$ for some $\mu\in\mathrm{Irr}(A/B)$ with $\mu\neq 1.$ Since $H$ centralizes $A/B$ and therefore fixes $\mu,$ we have
$$\mathrm{tr}_H(\lambda)=\sum_{h\in H}\lambda^h=\sum_{h\in H}\lambda^{gh}=\sum_{h\in H}(\lambda\mu)^h=\sum_{h\in H}\lambda^h\mu=\bigl(\mathrm{tr}_H(\lambda)\bigr)\mu.$$ Hence if $\mathrm{tr}_H(\lambda)(x)\neq 0$ then $x\in\mathrm{ker}(\mu)=B.$
\end{proof}

\begin{lem}\label{sigmalem2}
Suppose $A$ and $H$ are abelian $p$-groups with $H$ acting on $A.$ Let $\lambda\in\mathrm{Irr}(A)$ be a character of $A,$ and suppose $B$ is a $H$-invariant subgroup of $A$ with $|A:B|=p.$ Then $$\sigma(A,\lambda,H)\leq\sigma(B,\lambda_B,H)+(|A|-|B|)|C_H(\lambda)|.$$
\end{lem}

\begin{proof}
The proof is by induction on $|A|.$ First, assume that $B$ is the unique $H$-invariant subgroup of $A$ of index $p.$ If $x\in A,$ then since $H$ is abelian, $C_A(C_H(x))$ is a $H$-invariant subgroup of $A$ containing $x.$ As $A$ and $H$ are $p$-groups, every proper $H$-invariant subgroup of $A$ is contained in a $H$-invariant subgroup of index $p.$ Hence in this case $C_H(x)=C_H(A)$ for $x\in A\backslash B.$ Then $$\sigma(A,\lambda,H)\leq\sigma(B,\lambda_B,H)+(|A|-|B|)|C_H(A)|,$$ and the result follows since $C_H(A)\subseteq C_H(\lambda).$

We may therefore choose another $H$-invariant subgroup of index $p$ in $A.$ Let this be $B_1$ and set $C=B\cap B_1.$ Then $H$ acts trivially on $A/C,$ so the $p+1$ proper subgroups of $A$ strictly containing $C$ are all $H$-invariant. Two of these are $B$ and $B_1,$ and we label the others as $B_2,\dots,B_p.$ Then $A\backslash B$ is the disjoint union of the sets $B_1\backslash C,\dots,B_p\backslash C.$ Hence $$\sigma(A,\lambda,H)-\sigma(B,\lambda_B,H)=\sum_{i=1}^p\bigl(\sigma(B_i,\lambda_{B_i},H)-\sigma(C,\lambda_C,H)\bigr).$$

We first consider this equation assuming case $(1)$ of Lemma \ref{sigmalem1} holds for $B_i$ for each $i$ with $1\leq i\leq p.$ Then $C_H(\lambda_{B_i})=C_H(\lambda)$ for $1\leq i\leq p.$ Since $|B_i:C|=p$ for each $i,$ induction gives $$\sigma(B_i,\lambda_{B_i},H)-\sigma(C,\lambda_C,H)\leq(|B_i|-|C|)|C_H(\lambda_{B_i})|=(|B_i|-|C|)|C_H(\lambda)|.$$ The required result follows by summing over $i$ with $1\leq i\leq p.$

Now we may assume that case $(2)$ of Lemma \ref{sigmalem1} holds for $B_i$ for some $i$ with $1\leq i\leq p.$ Then $\mathrm{tr}_H(\lambda)$ vanishes off $B_i,$ so for $j$ with $1\leq j\leq p$ and $j\neq i,$ $\sigma(B_j,\lambda_{B_j},H)=\sigma(C,\lambda_C,H).$ By induction, $\sigma(B_i,\lambda_{B_i},H)-\sigma(C,\lambda_C,H)\leq(|B_i|-|C|)|C_H(\lambda_{B_i})|,$ and by Lemma \ref{sigmalem1}, $|C_H(\lambda_{B_i})|=p|C_H(\lambda)|.$ Hence 
$$\sigma(A,\lambda,H)-\sigma(B,\lambda_B,H)\leq p(|B_i|-|C|)|C_H(\lambda)|=(|A|-|B|)|C_H(\lambda)|.$$
This completes the proof.
\end{proof}

\begin{proof}[Proof of Proposition \ref{sigmaineq}]
We proceed by induction on $|A|.$ We can clearly assume $A\neq 1,$ so we may choose a $H$-invariant subgroup $B$ with $|A:B|=p.$ By induction, $\sigma(B,\lambda_B,H)\leq|B||C_H(\lambda_B)|.$ If case $(2)$ of Lemma \ref{sigmalem1} applies to $B,$ then $\sigma(A,\lambda,H)=\sigma(B,\lambda_B,H)$ and $|B||C_H(\lambda_B)|=|A||C_H(\lambda)|,$ and the result follows. Otherwise, case $(1)$ of Lemma \ref{sigmalem1} applies to $B.$  By Lemma \ref{sigmalem1}, $C_H(\lambda_B)=C_H(\lambda).$ By Lemma \ref{sigmalem2}, $$\sigma(A,\lambda,H)\leq\sigma(B,\lambda_B,H)+(|A|-|B|)|C_H(\lambda)|.$$ Hence $$\sigma(A,\lambda,H)\leq|B||C_H(\lambda)|+(|A|-|B|)|C_H(\lambda)|=|A||C_H(\lambda)|,$$ as required. This completes the proof.
\end{proof}

\begin{proof}[Proof of Theorem A$(1)$]
We write $k\bigl(\mathrm{Supp}(\chi)\bigr)$ to denote the number of conjugacy classes in $\mathrm{Supp}(\chi).$ We must prove that $k\bigl(\mathrm{Supp}(\chi)\bigr)\leq |P|/\chi(1)^2.$ Let $K=\mathrm{ker}(\chi),$ so that $\chi$ is the lift to $P$ of a faithful character $\chi_0\in\mathrm{Irr}(P/K).$ Clearly $k\bigl(\mathrm{Supp}(\chi)\bigr)\leq|K|k\bigl(\mathrm{Supp}(\chi_0)\bigr),$ so it suffices to prove that $k\bigl(\mathrm{Supp}(\chi_0)\bigr)\leq|P/K|/\chi_0(1)^2.$ Hence, we can assume $\chi=\chi_0$ is faithful.

By Lemma \ref{nM}, $\chi=\lambda^P$ for some $A\vartriangleleft P$ with $P/A$ abelian and some linear character $\lambda\in\mathrm{Irr}(A).$ Since $\chi$ is faithful, $A$ is abelian. Let $H=P/A$ with its natural action by conjugation on $A,$ and let $\lambda$ be a constituent of $\chi_A.$ Then $\chi_A=\mathrm{tr}_H(\lambda)$ and the support of $\chi$ is contained in $A.$ Proposition \ref{sigmaineq} applies and $C_H(\lambda)=1$ since $\chi=\lambda^P$ is irreducible. Hence $\sigma(A,\lambda,H)\leq|A|.$ We compute
$$k\bigl(\mathrm{Supp}(\chi)\bigr)=\sum_{x\in A,\chi(x)\neq 0}\frac{|C_H(x)|}{|H|}=\frac{\sigma(A,\lambda,H)}{|H|}\leq\frac{|A|}{|H|}=\frac{|P|}{|H|^2}=\frac{|P|}{\chi(1)^2}.$$ This completes the proof of Theorem A$(1).$
\end{proof}

For the proof of Theorem A$(2),$ we need some properties of the fields generated by the values of characters of $p$-groups. Denote by $\mathbb Q(\chi)$ the field generated by the values of a character $\chi.$ Also, let $\zeta_n=e^{2\pi\mathrm{i}/n},$ and let $\mathrm{tr}_{K/L}$ be the relative trace from $K$ to $L$ where $K\supseteq L$ are fields (no confusion should arise with the different operator $\mathrm{tr}_H$ used above).

\begin{lem}\label{fieldlem}
Let $p$ be an odd prime and let $n\geq 1.$ Let $K$ be a field. Then 
\begin{enumerate}
\item $\mathbb Q(\zeta_p)\subseteq K\subseteq\mathbb Q(\zeta_{p^n})$ if and only if $K=\mathbb Q(\zeta_{p^r})$ for some $r,$ $1\leq r\leq n.$ 
\item If $r>s\geq 1$ and $\zeta$ is any primitive $(p^r)^\mathrm{th}$ root of unity, then $$\mathrm{tr}_{\mathbb Q(\zeta_{p^r})/\mathbb Q(\zeta_{p^s})}(\zeta)=0.$$
\end{enumerate}
\end{lem}

\begin{proof}
The fields $\mathbb Q(\zeta_{p^r})$ for $1\leq r\leq n$ clearly lie between $\mathbb Q(\zeta_p)$ and $\mathbb Q(\zeta_{p^n}).$ On the other hand, the Galois group of $\mathbb Q(\zeta_{p^n})/\mathbb Q(\zeta_p)$ is the subgroup of $(\mathbb Z/p^n\mathbb Z)^\times$ consisting of classes that are congruent to $1$ mod $p.$ As $p$ is odd it is cyclic, generated for example by the class of $1+p.$ Since a cyclic $p$-group has at most one subgroup of each order, it follows that there are no other intermediate fields. For part $(2),$ since $r>s,$ we have $\zeta^{p^s}\neq 1.$ Therefore $$\mathrm{tr}_{\mathbb Q(\zeta_{p^r})/\mathbb Q(\zeta_{p^s})}(\zeta)=\mathop{\sum_{n\in(\mathbb Z/p^r\mathbb Z)^\times}}_{n\equiv 1\text{ mod } p^s}\zeta^n=\sum_{j=0}^{p^{r-s}-1}\zeta^{1+p^sj}=\zeta\frac{\zeta^{p^r}-1}{\zeta^{p^s}-1}=0,$$ as required.
\end{proof}

\begin{lem}\label{qchilem}
Let $P$ be a $p$-group, where $p$ is odd. Let $\chi\in\mathrm{Irr}(P).$ Then $\mathbb Q(\chi)=\mathbb Q(\zeta_{p^r}),$ where $p^r\leq {|P|/\chi(1)^2}.$
\end{lem}

\begin{proof}
We can assume that $\chi\neq 1.$ Then $\mathbb Q(\zeta_p)\subseteq\mathbb Q(\chi)\subseteq\mathbb Q(\zeta_{p^n}),$ where $p^n=|P|.$ By Lemma \ref{fieldlem}, $\mathbb Q(\chi)=\mathbb Q(\zeta_{p^r})$ for some $r\leq n.$ Then $\chi$ has $|\mathbb Q(\chi):\mathbb Q|=p^{r-1}(p-1)$ Galois conjugates, so $p^{r-1}(p-1)\chi(1)^2\leq|P|-1.$ Since $|P|/\chi(1)^2$ is a power of $p,$ it follows that $p^r\leq |P|/\chi(1)^2$ as required.
\end{proof}

\begin{lem}\label{corefree}
Let $P$ be a $p$-group, where $p$ is odd. Suppose $\chi\in\mathrm{Irr}(P)$ is faithful. Let $x\in P$ and suppose $\chi(x)\neq 0.$ Let $C=\langle x\rangle.$ If the order of $x$ does not divide $|P|/\chi(1)^2$ then $C\cap C^g=1$ for some $g\in P.$
\end{lem}

\begin{proof}
Let $p^e$ be the order of $x.$ By Lemma \ref{fieldlem}$(1)$ and Lemma \ref{qchilem}, $\mathbb Q(\chi)=\mathbb Q(\zeta_{p^f})$ where $p^f\leq |P|/\chi(1)^2.$ We can write $$|\mathbb Q(\zeta_{p^e}):\mathbb Q(\zeta_{p^f})|\chi(x)=\mathrm{tr}_{\mathbb Q(\zeta_{p^e})/\mathbb Q(\zeta_{p^f})}\bigl(\chi(x)\bigr).$$ By hypothesis, $e>f$ and $\chi(x)$ is not zero. In view of Lemma \ref{fieldlem}$(2),$ $\chi(x)$ is not a sum of primitive $(p^e)^\mathrm{th}$ roots of unity. Hence, some $\lambda\in\mathrm{Irr}(C)$ lying under $\chi,$ is not faithful. Therefore, if $D$ is the unique subgroup of $C$ of order $p,$ we have $[\chi,1]_D>0,$ so $D$ cannot be normal in $P$ since $\chi$ is faithful. Since $C$ is cyclic, there must be $g\in P$ such that $C\cap C^g$ does not contain $D.$ Then $C\cap C^g=1,$ as required. \end{proof}

\begin{proof}[Proof of Theorem A$(2)$]
Let $C$ be the cyclic group generated by $x.$ We have to prove that $|C|\leq |P|/\chi(1)^{3/2}.$ If $K=\ker(\chi)>1$ then by induction on $|P|,$ $|CK/K|\leq |P/K|/\chi(1)^{3/2},$ and the result follows. Hence, we may assume $\chi$ is faithful. By Lemma \ref{corefree}, either $|C|$ divides $|P|/\chi(1)^2$ or there exists $g\in P$ with $C\cap C^g=1.$ We may assume the latter holds. By Lemma \ref{nM}, there is a normal subgroup $A\vartriangleleft P$ and a linear character $\lambda\in\mathrm{Irr}(A)$ with $\lambda^P=\chi.$ Since $\chi(x)\neq 0,$ $C\subseteq A.$ Then as $C\cap C^g=1,$ $|C|^2\leq |A|=|P|/\chi(1).$ Since $\chi(1)^2\leq |P|,$ it follows that $|C|^2\leq |P|^2/\chi(1)^3,$ as required.
\end{proof}

\section{Groups of nilpotence class 3}
In this section, we will prove Theorem B. The proof of parts $(3)$ and $(4)$ depends on an important theorem of I.M. Isaacs \cite{Isaacs1973}. Recall that if $G$ is a group with a normal subgroup $N,$ then $\chi\in\mathrm{Irr}(G)$ and $\varphi\in\mathrm{Irr}(N)$ are said to be fully ramified with respect to $G/N$ if $\chi_N=e\varphi$ with $e^2=|G:N|.$ Following \cite{Isaacs1973}, we call a set $(G,K,L,\theta,\varphi)$ of groups and characters a \emph{character five} if $K$ and $L$ are normal subgroups of the finite group $G$ with $L\subseteq K,$ and $\theta\in\mathrm{Irr}(K)$ and $\varphi\in\mathrm{Irr}(L)$ are $G$-invariant characters and are fully ramified with respect to $K/L.$ The following only includes what we need for the proof of Theorem B; see \cite{Isaacs1973} for the full statement.

\begin{thm}[I.M. Isaacs. See Theorem 9.1 in \cite{Isaacs1973}]\label{FRC}
Let $(G,K,L,\theta,\varphi)$ be a character five. Suppose $K/L$ is abelian and $|K:L|$ is odd. Then there is a subgroup $U\subseteq G,$ and a (reducible, in general) character $\Psi$ of $G,$ such that 
\begin{enumerate}
\item $KU=G$ and $K\cap U=L.$
\item $\Psi(1)^2=|K/L|$ and $\Psi$ has no zeros on $G.$
\item For each $\chi\in\mathrm{Irr}(G\vert\theta)$ there exists $\xi\in\mathrm{Irr}(U\vert\varphi)$ such that $\chi_U=\Psi_U\xi$ and $\xi^G=\bar\Psi\chi.$
\end{enumerate}
\end{thm}

The next lemma is essentially Problem 3.12 in \cite{Isaacs1976}.

\begin{lem}\label{lem1}
Let $G$ be a finite group and let $\chi\in\mathrm{Irr}(G).$ Suppose $N$ is a normal subgroup of $G$ such that $\chi_N$ is irreducible. Then for $x\in G,$ the following formula holds: $$|\chi(x)|^2=\frac{\chi(1)}{|N|}\sum_{g\in N}\chi([x,g]).$$
\end{lem}

\begin{proof}
Fix $x\in G$ and let $H=N\langle x\rangle.$ Clearly $\chi_H\in\mathrm{Irr}(H).$ Let $\rho$ be a representation of $\mathbb CH$ affording $\chi_H$ and let $\mathcal C_x\in Z(\mathbb CH)$ be the sum over the conjugacy class $x^H=x^N.$ Then $\rho(\mathcal C_x)=\omega(x)I,$ where $I$ is the identity matrix and $\omega(x)=\vert x^N\vert\chi(x)/\chi(1).$ Hence $\rho(x^{-1}\mathcal C_x)=\omega(x)\rho(x^{-1}).$ Taking the matrix trace of both sides of this equation gives $\omega(x)\chi(x^{-1})=\sum_{y\in x^N}\chi(x^{-1}y),$ which is equivalent to the required result.
\end{proof}

We also need the following, which will provide the implication $(3)\implies (4)$ in the proof of Theorem B.

\begin{lem}\label{tD}
Let $P$ be a $p$-group, where $p$ is any prime. Let $\chi\in\mathrm{Irr}(P)$ and let $H$ be a subgroup of $P.$ Suppose $\chi$ is not zero at any element of $H.$ Then $|H|$ divides $|P|/\chi(1)^2.$
\end{lem}

\begin{proof}
Let $P=N_0\supset N_1\ldots\supset N_n=1$ be a chief series of $P.$ For $1\leq i\leq n,$ let $e_i=[\chi,\chi]_{N_i}/[\chi,\chi]_{N_{i-1}}.$ The $e_i$ are integers because $[\chi,\chi]_{N_i}$ is a power of $p$ for each $i$ with $0\leq i\leq n,$ and $[\chi,\chi]_{N_i}\geq [\chi,\chi]_{N_{i-1}}.$ Since $[\chi,\chi]_{N_i}\leq p[\chi,\chi]_{N_{i-1}},$ it follows that for each $1\leq i\leq n,$ either $e_i=1$ or $e_i=p,$ and also clearly $e_i=p$ if and only if $\chi$ vanishes on $N_{i-1}\backslash N_i.$ 

Now let $f_i=|HN_{i-1}:HN_i|.$ Then $f_i=p$ if $H\cap N_i=H\cap N_{i-1}$ and $f_i=1$ otherwise. Since $\chi$ does not vanish on $H$, $e_i=p$ only if $f_i=p.$ In particular, $e_i$ divides $f_i$ for each $i,1\leq i\leq n.$ Hence $\chi(1)^2=\prod_{i=1}^ne_i$ divides $|P:H|=\prod_{i=1}^nf_i,$  which is the required result.
\end{proof}

\begin{proof}[Proof of Theorem B]
We first show that $(1)$ implies $(2)$ and $(3)$ implies $(4).$ Let $k$ be the number of conjugacy classes in $\mathrm{Supp}(\chi)$ and let $x_i,$ for $1\leq i\leq k$ be representatives of those classes. Let $h_i$ be the size of the conjugacy class of $x_i.$ As $\chi$ is faithful, assuming $(1)$ gives $h_i|\chi(x_i)|^2=\chi(1)^2$ for each $i.$ Hence $|P|=\sum_ih_i|\chi(x_i)|^2=k\chi(1)^2,$ so $k=|P|/\chi(1)^2$ which is $(2).$ Also $(4)$ follows from $(3)$ and Lemma \ref{tD} applied to the cyclic group generated by $x.$

To complete the proof of Theorem B, we need to prove parts $(1)$ and $(3).$ We can clearly assume $\chi$ is faithful. Since $P$ has nilpotence class $3,$ it is metabelian, so by Lemma \ref{nM} we may choose a normal subgroup $A\vartriangleleft P$ with $P/A$ abelian and a linear character $\lambda\in\mathrm{Irr}(A)$ such that $\lambda^G=\chi.$ Since $\chi$ is faithful, $A$ is abelian. Note that $x\in A$ since $\chi$ vanishes off $A.$

Let $B=\langle x^P\rangle$ and $X=C_P(\lambda_B).$ Then $X\supseteq A,$ so $X\vartriangleleft P$ and hence also $[B,X]\vartriangleleft P.$ Also $\lambda([B,X])=1$ by definition of $X,$ and we conclude that $[B,X]\subseteq\ker\chi.$ Since $\chi$ is faithful, this implies that $X$ centralizes $B.$ Hence, writing $H=P/X,$ we can form the semidirect product $P_0=BH$ corresponding to the natural action of $H$ on $B.$ 

Let $\chi_0=(\lambda_B)^{P_0}.$ Since $X=C_P(\lambda_B)$ we see that $\chi_0\in\mathrm{Irr}(P_0).$ Also $(\chi_0)_B/\chi_0(1)=\chi_B/\chi(1);$ in particular, this shows that $\chi_0$ is faithful. Also, $h(x)=|H|$ whether computed in $P$ or $P_0.$ Hence if parts $(1)$ and $(3)$ of Theorem B are true for $P_0,$ $\chi_0$ and $x$ (regarded as an element of $P_0$), then those statements also hold for $P,$ $\chi$ and $x.$ We can therefore work with $P_0$ and $\chi_0$ for the remainder of the proof. For ease of notation, we will write $P_0=P$ and $\chi_0=\chi$ from now on.

Set $N=[B,H]H. $ Then $N\vartriangleleft P.$ Since $B=\langle x^H\rangle,$ it follows that $P=BH=\langle x\rangle N.$ Let $\varphi\in\mathrm{Irr}(N)$ be a constituent of $\chi_N$ and let $J$ be the stabilizer of $\varphi.$ Then $N\subseteq J\vartriangleleft G$ and $\chi$ is induced from a character of $J;$ since $\chi(x)\neq 0$ this implies $x\in J$ and so $J=P.$ Therefore $\chi_P$ is homogeneous. Since $P/N$ is cyclic, it follows that $\chi_N$ is irreducible. We have $[N,N]=[B,H,H],$ so since $P$ has nilpotence class $3,$ $[N,N]\subseteq Z(P).$ Since $\chi_N\in\mathrm{Irr}(N)$ is faithful, $\chi_N$ is fully ramified over $Z(P)$ and in particular, $\chi_N$ vanishes on $N\backslash Z(P)$ (see Theorem 2.31 in \cite{Isaacs1976}). Then by Lemma \ref{lem1}, $$|\chi(x)|^2=\frac{\chi(1)}{|N|}\sum_{g\in Q}\chi([x,g]),$$ where $Q$ is the pre-image in $N$ of $C_{N/Z(P)}(xZ(P)).$ Let $\mu$ be the map which takes $g\in Q$ to $\chi([x,g])/\chi(1).$ It is clear that $\mu$ is a linear character of $Q,$ and $$|\chi(x)|^2=\frac{\chi(1)^2|Q|}{|N|}[\mu,1]_Q.$$
Since $\chi(x)\neq 0,$ we must have $\mu=1,$ so $|\chi(x)|^2=\chi(1)^2/|N:Q|.$ Since $\chi$ is faithful, $\mu=1$ implies that $[x,g]=1$ for all $g\in Q,$ so $Q=C_N(x).$ Since $N=HC_N(x),$ we have $|N:Q|=|H:C_H(x)|.$ We obtain $$|\chi(x)|^2=\chi(1)^2/|H:C_H(x)|=\chi(1)^2/h(x).$$ (We know that $h(x)=|H|,$ but we do not need this here.)

Let $p^n=|P|$ and $E=\mathbb Q(\zeta_{p^n})$ and let $\mathcal O$ be the ring of integers of $E.$ Let $a=\chi(x)$ and $p^r=\chi(1)^2/h(x).$ Then $a\in \mathcal O$ and we have shown so far that $a\bar a=p^r.$ In particular, $(a)(\bar a)=(p)^r$ where the brackets denote ideals generated by elements of $\mathcal O.$ However $(p)$ factorizes as $(p)=\pi^{p^{n-1}(p-1)},$ where $\pi=(1-\zeta_{p^n})$ is the unique ramified prime of $\mathcal O.$ (See for example page $9$ in \cite{Washington1997}.) Since $(a)$ divides $(p)^r,$ $(a)$ is a power of $\pi$ and as $\pi=\bar{\pi},$ also $(a)=(\bar a).$ Hence $a=\bar a\omega$ for some $\omega\in\mathcal O.$ We see that $|\omega^\sigma|=1$ for any automorphism $\sigma$ of $E,$ so $\omega$ is a root of unity (see Lemma $1.6$ in \cite{Washington1997}). Hence $a^2=p^r(a/\bar a)=p^r\omega$ and so $a=p^{r/2}\zeta$ where $\zeta$ is a root of unity having $\zeta^2=\omega.$ This completes the proof of part $(1).$

For part $(3),$ we assume that $p$ is odd. Let $\varphi$ be the unique irreducible constituent of $\chi_Z.$ Then $(P,N,Z,\chi_N,\varphi)$ is a character five and Theorem \ref{FRC} applies. By Theorem \ref{FRC}$(1)$ and $(3)$, there exists $U\subseteq P$ with $UN=P,$ $U\cap N=Z$ and $\chi_U=\Psi_U\xi$ for some $\xi\in\mathrm{Irr}(U).$ Since $\chi(1)^2=|N/Z|,$ Theorem \ref{FRC}(2) shows that $\xi(1)=1$ and $\chi_U$ has no zeros. Since also $\chi=\bar\Psi\xi^G$ by Theorem \ref{FRC}$(3),$ we see that $\mathrm{Supp}(\chi)$ is exactly the union of the conjugates of $U.$ Part $(3)$ follows, since if $x$ is contained in a conjugate of $U$ then so are its powers. This completes the proof of part $(3),$ and the proof of Theorem B is complete.
\end{proof}

\section{Groups of small order}
In this last section, we prove Theorem C.

\begin{lem}\label{ilem}
Let $N$ be a normal subgroup of the $p$-group $P.$ Suppose $\chi\in\mathrm{Irr}(P)$ and let $\varphi$ be an irreducible constituent of $\chi_N.$ If $x\in\mathrm{Supp}(\chi)$ then $x^m\in N,$ where $m=\frac{|P:N|}{\chi(1)/\varphi(1)}.$
\end{lem}

\begin{proof}
We use induction on $|P:N|.$ The result is clear when $N=P,$ so assume $N\neq P$ and let $M$ be a normal subgroup of $P$ with $N\subset M\subseteq P$ and $|M/N|=p.$ Let $\theta$ be an irreducible constituent of $\chi_M$ that lies over $\varphi.$ By induction, $x^r\in M,$ where $r=\frac{|P:M|}{\chi(1)/\theta(1)}.$ There are two possibilities; either $\theta_N=\varphi$ and $rp=m,$ or $\theta=\varphi^M$ and $r=m.$ In the first case, $x^r\in M$ immediately implies $x^m\in N$ as required. In the second case, let $J$ be the stabilizer of $\varphi\in G.$ Then $J\cap M\subseteq N.$ But $\chi(x)\neq 0,$ so $x$ must be contained in some conjugate $J^g$ of $J,$ and since also $J^g\cap M\subseteq N,$ it follows that $\langle x\rangle\cap M\subseteq N.$ Therefore $x^r=x^m\in N,$ as required.
\end{proof}

\begin{proof}[Proof of Theorem C]
We have a $p$-group $P$ with $\chi\in\mathrm{Irr}(P)$ and $x\in\mathrm{Supp}(\chi)$ such that the order of $x$ does not divide $|P|/\chi(1)^2,$ and we have to show that $\chi(1)\geq p^4$ and that $|P|\geq p^{10}.$ If $Z$ is the centre of $P$ then $\chi(1)^2$ divides $|P/Z|,$ so $x\notin Z,$ and in particular $\chi$ is not of central type and so $|P|/\chi(1)^2$ is at least $p^2.$ We therefore need only prove that $\chi(1)\geq p^4,$ since then $|P|\geq p^{8+2}=p^{10}.$

Let $C$ be the cyclic group generated by $x$ and write $$|P|=p^n,\chi(1)=p^a,|C|=p^e,$$ so our assumption is that $e\geq n-2a+1,$ yet $x\in\mathrm{Supp}(\chi).$ We may assume that $\chi(1)$ is as small as possible subject to these conditions, and we must show that $a\geq 4.$

If $K$ is the kernel of $\chi$ then the image $xK$ of $x$ in $P/K$ has order at least equal to $p^e/|K|$ and lies in $\mathrm{Supp}(\chi)$ where $\chi$ is regarded as a character of $P/K,$ so $P/K,\chi$ and $xK$ constitute another example with the same value $a.$ Hence, we can assume $\chi$ is faithful. 

Since $\chi$ is not linear, $\chi$ is induced from a character $\varphi$ of some maximal subgroup $M$ of $P.$ Since $\chi(x)\neq 0,$ $x\in M$ and $x^g\in\mathrm{Supp}(\varphi)$ for some $g\in P.$ Since $\varphi(1)<\chi(1),$ it follows that $p^e$ divides $|M|/\varphi(1)^2=p^{n-2a+1}.$ On the other hand $e>n-2a,$ so in fact $e=n-2a+1.$

Since $\chi$ is faithful, Lemma \ref{corefree} shows that there exists $h\in P$ such that $C\cap C^h=1.$ Hence $|C|^2\leq |M|,$ or $2e+1\leq n.$ Also $n-2a\geq 2$ because $\chi$ cannot be of central type, as remarked at the beginning of the proof. Hence $e=n-2a+1\geq 3$ and $n\geq 7.$ If $n=7$ then $e$ is even, so $e\geq 4$ and $n< 2e+1,$ which is not the case. Hence $n\geq 8.$

To complete the proof, we need normal subgroups $X$ and $Y$ of $P$ with $|X|=p^3,$ $|Y|=p^5$ and $X$ contained in the centre of $Y.$ By a classical result, if an abelian $p$-group $A$ has order $p^m$ then the $p$-part of $|\mathrm{Aut}(A)|$ divides $p^{m(m-1)/2}.$ (This can be found in \cite{Burnside}.) Let $A$ be any maximal abelian normal subgroup of $P.$ Then $A=C_P(A)$ and so if $|A|=p^m$ then $|P|\leq p^{m(m+1)/2}.$ Since $|P|>p^6,$ $m\geq 4.$ Choose $X$ to be any normal subgroup of $P$ contained in $A$ and of order $p^3.$ Now $|P:C_P(X)|\leq p^3,$ so since $|P|\geq p^8,$ $|C_P(X)|\geq p^5.$ Choose any $Y\vartriangleleft P$ with $X\subseteq Y\subseteq C_P(X)$ and $|Y|=p^5.$

Let $\theta\in\mathrm{Irr}(Y)$ be an irreducible constituent of $\chi_Y$ and let $\lambda$ be the unique irreducible constituent of $\theta_X.$ Either $\theta$ is linear and $\theta_X=\lambda$ or $\theta(1)=p$ and $\theta$ and $\lambda$ are fully ramified with respect to $Y/X.$ 

If the first of these possibilities holds, Lemma \ref{ilem} gives $D\subseteq Y,$ where $D=\langle x^{|P:Y|/\chi(1)}\rangle$ Since $D\cap D^h=1,$ it follows that $|D|^2\leq |Y|$ or $2(e-n+5+a)\leq 5.$ Thus $e+a\leq n-3$ and since $n-2a+1=e,$ we obtain $a\geq 4,$ as required. 

Assuming now the second possibility holds, let $J$ be the stabilizer of $\theta.$ Then $(J,Y,X,\theta,\lambda)$ is a character five. Since $Y/X$ is abelian of odd order, by Theorem \ref{FRC}$(1)$ and $(3)$ there exists a subgroup $U$ such that $YU=J,$ $Y\cap U=X$ and $\theta$ vanishes on elements of $J$ that are not conjugate to an element of $U.$ Some conjugate of $x$ must be contained in $J$ since $\chi$ is induced from a character of $J.$ In summary, some conjugate of $x$ lies in $U.$ It follows that $C\cap Y\subseteq X.$ Lemma \ref{ilem} gives $E\subseteq Y$ where $E=\langle x^{p|P:Y|/\chi(1)}\rangle,$ so $E\subseteq X.$ Since $E\cap E^h=1,$ $|E|^2\leq |X|$ and so $2(e-1-n+5+a)\leq 3,$ or $e+a\leq n-3.$ Again using $n-2a+1=e,$ it follows that $a\geq 4.$ This completes the proof.
\end{proof}

\emph{London, England. Email: tom@beech84.fsnet.co.uk.}

\end{document}